\documentclass[
10pt,
paper=a4,
headings=small,
english,
]{scrartcl}

\setkomafont{sectioning}{\rmfamily\bfseries}

\usepackage[T1]{fontenc}
\usepackage[utf8]{inputenc}
\usepackage{amsfonts,amsmath,amssymb,amsthm}
\usepackage{url,hyperref}
\usepackage{microtype}

\newcommand{\IQ}{\mathbb{Q}}

\newcommand{\IZ}{\mathbb{Z}}
\newcommand{\IC}{\mathbb{C}}

\newcommand{\IR}{\mathbb{R}}

\newcommand{\IA}{\mathbb{A}}

\renewcommand{\Re}{\mathrm{Re}}

\newcommand{\Borel}{B}
\newcommand{\Siegel}{P}
\newcommand{\Klingen}{Q}

\DeclareMathOperator{\St}{St}

\DeclareMathOperator{\GL}{GL}

\DeclareMathOperator{\GSp}{GSp}
\DeclareMathOperator{\GSO}{GSO}
\DeclareMathOperator{\GO}{GO}
\DeclareMathOperator{\Sp}{Sp}

\DeclareMathOperator{\simi}{sim}

\newtheorem{theorem}{Theorem}[section]
\newtheorem{lemma}[theorem]{Lemma}
\newtheorem{proposition}[theorem]{Proposition}

\newtheorem{hypothesis}[theorem]{Hypothesis}

\begin{document}
\title{
\large{Multiplicity one for certain paramodular forms of genus two}
}

\date{}
\author{\large{Mirko R\"osner} \and \large{Rainer Weissauer}}

\maketitle

\begin{abstract} 
We show that certain paramodular cuspidal automorphic irreducible representations of $\GSp(4,\IA_\IQ)$, which are not CAP, are globally generic. 
This implies a multiplicity one theorem for paramodular cuspidal automorphic representations.
Our proof relies on a reasonable hypothesis concerning the non-vanishing of central values of automorphic $L$-series.
\let\thefootnote\relax
\let\thefootnote\relax
\end{abstract}

\section{Introduction}
Atkin-Lehner theory defines a one-to-one correspondence between cuspidal automorphic irreducible representations of $\GL(2,\IA_\IQ)$ with archimedean factor in the discrete series and normalized holomorphic elliptic cuspidal newforms on the upper half plane, that are eigenforms for the Hecke algebra.
As an analogue for the symplectic group $\GSp(4,\IA_\IQ)$, a local theory of newforms has been developed by Roberts and Schmidt \cite{Roberts-Schmidt} with respect to paramodular groups. 

However, still lacking for this theory is the information whether paramodular cuspidal automorphic irreducible representations of $\GSp(4,\IA_\IQ)$ occur in the cuspidal spectrum with multiplicity one. Furthermore, holomorphic paramodular cusp forms, i.e.\ those invariant under some paramodular subgroup of $\Sp(4,\IQ)$, do not describe all holomorphic Siegel modular cusp forms.
Indeed, at least if the weight of the modular forms is high enough, one is lead to conjecture that the paramodular holomorphic cusp forms exactly correspond to those holomorphic modular cusp forms for which their local non-archimedean 
representations, considered from an automorphic point of view, are generic representations. 
Under certain technical restrictions, we show that this is indeed the case. 

To be more precise,
suppose $\Pi=\bigotimes_v\Pi_v$ is a paramodular cuspidal automorphic irreducible representation
of $\GSp(4,\IA_\IQ)$ of odd paramodular level, which is not CAP and whose archimedean factor $\Pi_\infty$
is in the discrete series. Then under the assumption of the hypothesis below we prove that
the local representations $\Pi_v$ are generic at all non-archimedean places $v$.
The restriction to odd paramodular level comes from Danisman \cite{Danisman} and can probably be removed.
Furthermore, we show that the hypothesis implies that $\Pi$ occurs in the cuspidal spectrum with multiplicity one
and is uniquely determined by almost all of its local factor representations.
The hypothesis imposed concerns the non-vanishing of central $L$-values and is crucial for our approach.

\begin{hypothesis}\label{conj:non-vanishing-L-value}
Suppose $\Pi$ is a globally generic unitary cuspidal automorphic irreducible representation of $\GSp(4,\IA_\IQ)$ and $\alpha$ and $\beta>0$ are real numbers.
Then there is a unitary idele class character $\mu:\IQ^\times\backslash\IA^\times_\IQ\to\IC^\times$, locally trivial at a prescribed non-archimedean place of $\IQ$, such that the twisted Novodvorsky L-function
\begin{equation} L_{\mathrm{Nvd}}(\Pi,\mu,s)\end{equation}
does not vanish at $s=1/2+i(\alpha+k\beta)$ for some integer $k$.
\end{hypothesis}
The analogous hypothesis for the group $\GL(4)$ would imply our hypothesis, see Prop.\,\ref{prop:lift_gsp4_gl4}. 
The corresponding statement for $\GL(2)$ is well-known \cite[Thm.\,4]{Waldspurger_Correspondence_de_Shimura}.
For $\GL(r)$, $r=1,2,3$, compare \cite{Friedberg-Hoffstein}, \cite{Hoffstein_Kantorovich}.
An approximative result for $\GL(4)$ has been shown by Barthel and Ramakrishnan
\cite{Barthel_Rama_Non_vanishing_of_L-functions}, later improved by Luo \cite{Luo_Non_vanishing_of_L-functions}:
Given a unitary globally generic cuspidal automorphic irreducible representation $\Pi$ of $\GL(4,\IA_\IQ)$,
a finite set $S$ of $\IQ$-places and a complex number $s_0$ with $\Re(s_0)\neq1/2$ there are infinitely many Dirichlet characters $\mu$ such that $\mu_v$ is unramified for $v\in S$ and the completed $L$-function $\Lambda((\mu\circ\det)\otimes\Pi,s)$ does not vanish at $s=s_0$.


We remark, there is good evidence for our result (Thm.\,\ref{thm:stable_para_implies_glob_generic}) on genericity of paramodular representations. In fact, the generalized strong Ramanujan conjecture for cuspidal automorphic irreducible representations $\Pi=\otimes'_v\Pi_v$ of $\GSp(4,\IA_\IQ)$ (not CAP) predicts that every local representation $\Pi_v$ should be tempered. But paramodular tempered local representations $\Pi_v$ at non-archimedean places are always generic by Lemma \ref{RS-temps_are_param_iff_gen}.

\section{Preliminaries}
The group $\mathbf{G}=\GSp(4)$ (symplectic similitudes  of genus two) is defined over $\mathbb{Z}$ by the equation 
\begin{equation*}g^tJg=\lambda J\end{equation*} for $(g,\lambda)\in \GL(4)\times\GL(1)$ and $J=\left(\begin{smallmatrix}0& w\\- w&0\end{smallmatrix}\right)$ with $w=\left(\begin{smallmatrix}0& 1\\1 &0\end{smallmatrix}\right)$. Since $\lambda$ is uniquely determined by $g$, we write $g$ for $(g,\lambda)$ and obtain the similitude character
\begin{equation*}\simi:\GSp(4)\to \GL(1)\,,\quad g\mapsto \lambda.
\end{equation*}

Fix a totally real number field $F/\IQ$ with integers $\mathfrak{o}$ and adele ring $\IA_F=\IA_\infty\times\IA_{\mathrm{fin}}$. 
For the profinite completion of $\mathfrak{o}$ we write $\mathfrak{o}_{\mathrm{fin}}\subseteq \IA_{\mathrm{fin}}$.
The paramodular group $K^{\mathrm{para}}(\mathfrak{a})\subseteq \GSp(4,\IA_{\mathrm{fin}})$ attached to a non-zero ideal $\mathfrak{a}\subseteq\mathfrak{o}$ is the group of all 
\begin{align*}
g\in
\begin{pmatrix}
\mathfrak{o}_{\mathrm{fin}} & \mathfrak{o}_{\mathrm{fin}} &\mathfrak{o}_{\mathrm{fin}}  & \mathfrak{a}^{-1}\mathfrak{o}_{\mathrm{fin}}\\
\mathfrak{a}\mathfrak{o}_{\mathrm{fin}} & \mathfrak{o}_{\mathrm{fin}} &\mathfrak{o}_{\mathrm{fin}}  & \mathfrak{o}_{\mathrm{fin}}  \\
\mathfrak{a}\mathfrak{o}_{\mathrm{fin}} & \mathfrak{o}_{\mathrm{fin}} &\mathfrak{o}_{\mathrm{fin}}  & \mathfrak{o}_{\mathrm{fin}}  \\
\mathfrak{a}\mathfrak{o}_{\mathrm{fin}} & \mathfrak{a}\mathfrak{o}_{\mathrm{fin}} &\mathfrak{a}\mathfrak{o}_{\mathrm{fin}}  & \mathfrak{o}_{\mathrm{fin}}  \\
\end{pmatrix} \cap \GSp(4,\IA_{\mathrm{fin}}),\qquad \simi(g)\in\mathfrak{o}_\mathrm{fin}^{\times}.
\end{align*} 

An irreducible smooth representation $\Pi=\Pi_\infty\otimes\Pi_{\mathrm{fin}}$ of $\GSp(4,\IA_F)$ is called paramodular
if $\Pi_{\mathrm{fin}}$ admits non-zero invariants under some paramodular group $K^{\mathrm{para}}(\mathfrak{a})$.

Two irreducible automorphic representations are said to be weakly equivalent if they are locally isomorphic at almost every place.
A cuspidal automorphic irreducible representation of $\GSp(4)$ is CAP if it is weakly equivalent to a constituent
of a globally parabolically induced representation from a cuspidal representation of the Levi quotient of a proper parabolic subgroup.
In that case we say that $\Pi$ is strongly associated to this parabolic.
The standard Borel $\mathbf{\Borel}$, Siegel parahoric $\mathbf{\Siegel}$ and Klingen parabolic $\mathbf{\Klingen}$ are
\begin{equation*}
  \mathbf{\Borel}=\left(
\begin{smallmatrix}\ast&\ast&\ast&\ast\\&\ast&\ast&\ast\\&&\ast&\ast\\&&&\ast\end{smallmatrix}
\right)\cap \mathbf{G}\ ,\qquad 
 \mathbf{\Siegel}=\left(
\begin{smallmatrix}\ast&\ast&\ast&\ast\\\ast&\ast&\ast&\ast\\&&\ast&\ast\\&&\ast&\ast\end{smallmatrix}
\right)\cap \mathbf{G}\ ,\qquad 
\mathbf{\Klingen}=\left(
\begin{smallmatrix}\ast&\ast&\ast&\ast\\&\ast&\ast&\ast\\&\ast&\ast&\ast\\&&&\ast\end{smallmatrix}
\right)\cap \mathbf{G}\ .
\end{equation*}

\section{Poles of local spinor factors}
Fix a local nonarchimedean place $v$ of $F$ with completion $F_v$, valuation character $\nu(x)=|x|_v$ for $x\in F_v$, residue field $\mathfrak{o}_v/\mathfrak{p}_v$ of odd order $q$ and uniformizer $\varpi\in\mathfrak{p}_v$. In this section we consider preunitary irreducible admissible representations $\Pi_v$ of $G=\mathbf{G}(F_v)$.
Such a $\Pi_v$ is called paramodular if it admits non-zero invariants under the local factor $K_v^{\mathrm{para}}(\mathfrak{a})$ of the paramodular group attached to some non-zero ideal $\mathfrak{a}$ of the integers $\mathfrak{o}$.
\begin{lemma}\label{lem:param_trivial_cent_char_up_to_unramif_twist}
For every paramodular $\Pi_v$, there is an unramified character $\omega'_v$ of $F_v^\times$, such that
the twist $(\omega'_v\circ \simi)\otimes\Pi_v$ has trivial central character.
\end{lemma}
\begin{proof}
The intersection of $K_v^\mathrm{para}(\mathfrak{a})$ with the center of $G$ is isomorphic to $\mathfrak{o}_v^\times$, so the central character of $\Pi_v$ is unramified.
\end{proof}
\begin{lemma}\label{RS-temps_are_param_iff_gen}
For tempered preunitary irreducible admissible representations $\Pi_v$ the following assertions are equivalent:
\begin{enumerate}
\item[i)] $\Pi_v$ is generic and has unramified central character,
\item[ii)] $\Pi_v$ is paramodular.
\end{enumerate}
\end{lemma}
\begin{proof}
By Lemma~\ref{lem:param_trivial_cent_char_up_to_unramif_twist}, we can assume that $\Pi_v$ has trivial central character. Then this is a result of Roberts and Schmidt \cite[7.5.8]{Roberts-Schmidt}.
\end{proof}
Recall that for every smooth character $\chi:F_v^\times\to\mathbb{C}^\times$ the local Tate $L$-factor is
\begin{equation*}
 L(\chi,s)=\begin{cases}  (1-\chi(\varpi)q^{-s})^{-1}&\chi\ \text{unramified,}\\   1&\chi\ \text{ramified.}  \end{cases}
\end{equation*}

For a generic irreducible admissible representation $\Pi_v$ of $G$ and a smooth complex character $\mu$ of $F_v^\times$, Novodvorsky \cite{Novod_L_factors} has defined a local degree four spinor $L$-factor $L_{\mathrm{Nvd}}(\Pi_v,\mu,s)$. Piatetskii-Shapiro and Soudry \cite{PS-L-Factor_GSp4}, \cite{Soudry_Piatetski_Weil_lift} have extended this construction to a local degree four spinor $L$-factor $L_{\mathrm{PSS}}(\Pi_v,\mu,s)$ for infinite-dimensional irreducible admissible representations of $G$. 

We write $L_{\mathrm{PSS}}(\mu\Pi_v,s)$ for $L_{\mathrm{PSS}}(\Pi_v,\mu,s)=L_{\mathrm{PSS}}(\mu\Pi_v,1,s)$.
Poles of $L_{\mathrm{PSS}}(\Pi_v,s)$ are called regular if they occur as poles of certain zeta integrals \cite[\S 2]{Soudry_Piatetski_Weil_lift}; the other poles are exceptional.

\begin{lemma}\label{Lem:Factors_of_PS_Nov_equiv}
Suppose the residue characteristic of $F_v$ is odd.
For every generic irreducible admissible representation $\Pi_v$ the local
factors $L_{\mathrm{Nvd}}(\Pi_v,\mu,s)$ and $L_{\mathrm{PSS}}(\Pi_v,\mu,s)$ coincide.
\end{lemma}
\begin{proof}The local $L$-factors are products of Tate factors, so they are uniquely determined by their regular poles \cite[Thm.\,4.4]{PS-L-Factor_GSp4}. The local factor $L_{\mathrm{PSS}}(\Pi_v,\mu,s)$ is trivial for generic cuspidal $\Pi_v$ \cite[\S5.1]{Danisman}, \cite[4.3]{PS-L-Factor_GSp4}. For non-cuspidal $\Pi_v$, the regular poles have been determined explicitly by Danisman \cite{Danisman}, \cite{Danisman2}. For the Tate factors of $L_{\mathrm{Nvd}}(\Pi_v,s)$, see Takloo-Bighash \cite{TB_L-factors}.
\end{proof}

The non-cuspidal $\Pi_v$ have been classified by Sally and Tadic \cite{Sally_Tadic} and we use their notation. Roberts and Schmidt \cite{Roberts-Schmidt} have designated them with roman numerals.

\begin{lemma}\label{lem:reg_poles} Suppose the residue characteristic of $F_v$ is odd. Let $\Pi_v$ be a preunitary non-generic irreducible admissible representations of $G$, that is not one-dimensional. Then $L_{\mathrm{PSS}}(\Pi_v,s)$ has a regular pole on the line $\Re(s)=1/2$ exactly in the following cases:
\begin{enumerate}
\item[IIb] $\Pi_v\cong(\chi\circ\det)\rtimes\sigma$                    for a pair of characters $\chi,\sigma$ that are either both unitary or satisfy $\chi^2=\nu^{2\beta}$ for $0<\beta<\tfrac{1}{2}$ with unitary $\chi\sigma$.
$L_{\mathrm{PSS}}(s,\Pi_v)$ contains the Tate factor $L(s,\nu^{-1/2}\chi\sigma)$, so poles with $\Re(s)=1/2$ occur if and only if $\chi\sigma$ is unramified.
\item[IIIb] $\Pi_v\cong\chi\rtimes(\sigma\circ\det)$                   for unitary characters $\sigma$ and $\chi$ with $\chi\neq1$. 
The regular poles with $\Re(s)=1/2$ come from the Tate factors $L(s,\nu^{-1/2}\sigma)$ and $L(s,\nu^{-1/2}\sigma\chi)$ in $L_{\mathrm{PSS}}(s,\Pi_v)$, so they occur for unramified $\sigma$ or $\sigma\chi$, respectively.
\item[Vb,c] $\Pi_v\cong L(\nu^{1/2}\xi\St_{\GL(2)},\nu^{-1/2}\sigma)$  for unitary characters $\sigma$ and $\xi$ with $\xi^2=1\neq\xi$.
The regular poles with $\Re(s)=1/2$ come from the Tate factor $L(s,\nu^{-1/2}\sigma)$ and appear for unramified $\sigma$.
\item[Vd] $\Pi_v\cong L(\nu\xi,\xi\rtimes\nu^{-1/2}\sigma)$            for unitary characters $\sigma$ and $\xi$ with $\xi^2=1\neq\xi$.
The regular poles with $\Re(s)=1/2$ come from the Tate factors $L(s,\nu^{-1/2}\sigma)$ and $L(s,\nu^{-1/2}\xi\sigma)$, and occur for unramified $\sigma$ or $\xi\sigma$, respectively.
\item[VIc] $\Pi_v\cong L(\nu^{1/2}\St_{\GL(2)},\nu^{-1/2}\sigma)$      for unitary $\sigma$.
The Tate factor $L(s,\nu^{-1/2}\sigma)$ gives rise to regular poles with $\Re(s)=1/2$ when $\sigma$ is unramified.
\item[VId] $\Pi_v\cong L(\nu,1\rtimes \nu^{-1/2}\sigma)$               for unitary $\sigma$. The Tate factor $L(s,\nu^{-1/2}\sigma)^2$ gives rise to double regular poles with $\Re(s)=1/2$ when $\sigma$ is unramified.
\item[XIb] $\Pi_v\cong L(\nu^{1/2}\pi,\nu^{-1/2}\sigma)$, where $\pi$ is a preunitary cuspidal irreducible admissible representation of $\GL(2,F_v)$ with trivial central character and $\sigma$ is a unitary character. The regular poles with $\Re(s)=1/2$ occur with the Tate factor $L(s,\nu^{-1/2}\sigma)$ when $\sigma$ is unramified.
\end{enumerate}
\end{lemma}

\begin{proof} For the non-cuspidal case, see Danisman \cite[Table A.1]{Danisman}, \cite[Table A.1]{Danisman2}.
For cuspidal irreducible $\Pi_v$, there are no regular poles \cite[\S5.1]{Danisman}.
\end{proof}

\begin{lemma}\label{lem:para} 
A non-generic preunitary irreducible admissible representation $\Pi_v$ of $G$ is paramodular if and only if it is isomorphic to one of the following:
\begin{enumerate}
\item[IIb] $(\chi\circ\det)\rtimes\sigma$,                  for characters $\chi,\sigma$ such that $\chi\sigma$ is unramified and either both are unitary or $\chi^2=\nu^{2\beta}$ for $0<\beta<\tfrac{1}{2}$ with unitary characters $\chi\sigma$,
\item[IIIb] $\chi\rtimes(\sigma\circ\det)$,                 for unramified unitary characters $\chi,\sigma$ with $\chi\neq1$,
\item[IVd] $\sigma\circ \simi$,                                 for unramified unitary characters $\sigma$,
\item[Vb,c] $L(\nu^{1/2}\xi\St_{\GL(2)},\nu^{-1/2}\sigma)$, for $\xi$ with $\xi^2=1\neq\xi$ and unramified unitary $\sigma$,
\item[Vd] $L(\nu\xi,\xi\rtimes\nu^{-1/2}\sigma)$,           for unramified unitary characters $\sigma,\xi$ with $\xi^2=1\neq\xi$,
\item[VIc] $L(\nu^{1/2}\St,\nu^{-1/2}\sigma)$,              for unramified unitary characters $\sigma$,
\item[VId] $L(\nu,1\rtimes \nu^{-1/2}\sigma)$,              for unramified unitary characters $\sigma$,
\item[XIb] $L(\nu^{1/2}\pi,\nu^{-1/2}\sigma)$,              for a cuspidal preunitary irreducible admissible representation $\pi$ of $\GL(2,F_v)$ with trivial central character and an unramified unitary $\sigma$.
\end{enumerate}
\end{lemma}
\begin{proof} By Lemma \ref{lem:param_trivial_cent_char_up_to_unramif_twist}, we can assume that the central character is trivial. For non-cuspidal $\Pi_v$, see Tables A.2 and A.12 of Roberts and Schmidt \cite{Roberts-Schmidt}. Cuspidal non-generic preunitary representations are never paramodular by Lemma \ref{RS-temps_are_param_iff_gen}.
\end{proof}

\begin{proposition}\label{prop:local_params_have_pole_at_one_half} Suppose the residue characteristic of $F_v$ is odd.
Let $\Pi_v$ be a paramodular preunitary irreducible admissible representation of $G$, that is not one-dimensional. The following assertions are equivalent:
\begin{itemize}
\item[i)] $\Pi_v$ is non-generic,
\item[ii)] the spinor $L$-factor $L_{\mathrm{PSS}}(\Pi_v,s)$ has at least one pole on the line $\Re(s)=1/2$.
\end{itemize}
\end{proposition}
\begin{proof}
By the previous two lemmas, for non-generic $\Pi_v$ there is a regular pole on the line $\Re(s)=1/2$.
If $\Pi_v$ is generic, the Euler factor $L_{\mathrm{PSS}}(\Pi_v,s)$ has only regular poles \cite[4.3]{PS-L-Factor_GSp4} and these do not occur on the line $\Re(s)=1/2$ \cite{TB_L-factors}.
\end{proof}

\pagebreak

\section{Global genericity}
Let $F=\IQ$ with adele ring $\IA=\IR\times\IA_{\mathrm{fin}}$. In the following, suppose $\Pi = \Pi_\infty \otimes \Pi_{\mathrm{fin}}$ is a cuspidal automorphic irreducible representation of $\GSp(4,\IA)$, not CAP, with central character $\omega_\Pi$, such that $\Pi_\infty$ belongs to the discrete series. We want to show that if $\Pi$ is paramodular, then $\Pi_v$ is locally generic at every nonarchimedean place $v$.

Recall that the product $L_{\mathrm{PSS}}(\Pi,s)=\prod_v L_{\mathrm{PSS}}(\Pi_v,s)$ converges for $s$ in a right half plane and admits a meromorphic continuation to $\IC$ \cite[Thm.\,5.3]{PS-L-Factor_GSp4}. This is the global degree four spinor $L$-series of Piatetskii-Shapiro and Soudry. It satisfies a functional equation
\begin{equation}
 L_{\mathrm{PSS}}(\Pi,s)=\epsilon(\Pi,s)L_{\mathrm{PSS}}(\Pi^\vee,1-s)
\end{equation} where $\Pi^\vee\cong\Pi\otimes(\omega_\Pi^{-1}\circ \simi)$ is the contragredient. 

\begin{proposition}[Generalized Ramanujan]\label{weak_Ramanujan}
 The spherical local factors $\Pi_v$ of $\Pi_{\mathrm{fin}}$ are isomorphic to irreducible tempered principal series representations $\chi_1\times\chi_2\rtimes\sigma$ for  unramified unitary complex characters $\chi_1,\chi_2,\sigma$ of $\IQ_v$.
\end{proposition}
\begin{proof}
 See \cite[Thm.\,3.3]{Weissauer200903}.
\end{proof}

\begin{proposition}\label{prop:weak_lift_to_globally_generic}
$\Pi$ is weakly equivalent to a unique globally generic cuspidal automorphic irreducible representation 
$\Pi_{\mathrm{gen}}$ of $\GSp(4,\IA)$ whose archimedean local component $\Pi_{\mathrm{gen},\infty}$ is in the local archimedean $L$-packet of $\Pi_{\infty}$. The lift $\Pi\mapsto\Pi_{\mathrm{gen}}$ commutes with character twists by unitary idele class characters. The central characters of $\Pi_{\mathrm{gen}}$ and $\Pi$ coincide.
\end{proposition}
\begin{proof} See \cite[Thm.\,1]{Weissauer_Schiermonnikoog}; the proof relies on certain Hypotheses A and B shown in \cite{Weissauer200903}. The lift commutes with twists because $\Pi_{\mathrm {gen}}$ is unique. The central characters are weakly equivalent, so they coincide globally by strong multiplicity one for $\GL(1,\IA)$.
\end{proof}

\begin{proposition}\label{prop:mult_in_4D-reps} If $\Pi$ is not a weak endoscopic lift, then $\Pi_\infty$ is contained in an archimedean local $L$-packet $\{ \Pi_\infty^+ , \Pi_\infty^- \}$ such that the multiplicities
of $\Pi_\infty^+ \otimes \Pi_{\mathrm{fin}}$ and $\Pi_\infty^- \otimes \Pi_{\mathrm{fin}}$ in the cuspidal spectrum coincide.
\end{proposition}
\begin{proof} By Prop.~\ref{prop:weak_lift_to_globally_generic}, $\Pi$ is weakly equivalent to a globally generic representation $\Pi'$ of $\GSp(4,\IA)$, which satisfies multiplicity one \cite{Jiang_Soudry_mult_one_GSp4}. Now \cite[Prop.\,1.5]{Weissauer_asterisque} implies the statement.\end{proof}

\begin{proposition}\label{prop:lift_gsp4_gl4}
Suppose $\Pi$ is globally generic. Then there is a unique globally generic automorphic irreducible representation $\tilde{\Pi}$ of $\GL(4,\IA)$ with partial Rankin-Selberg L-function
\begin{equation*}
 L^S(\tilde{\Pi},s)=L^S_{\mathrm{PSS}}(\Pi,s)
\end{equation*}
for a sufficiently large set $S$ of places. This lift is local in the sense that $\tilde{\Pi}_v$ only depends on $\Pi_v$. It commutes with character twists by unitary idele class characters.
\end{proposition}
\begin{proof}
For the existence and locality of the lift, see Asgari and Shahidi \cite{Asgari_Shahidi_GSp4_GL4}; uniqueness follows from strong multiplicity one for $\GL(4)$. It remains to be shown that $\Pi\mapsto\tilde{\Pi}$ commutes with character twists. Indeed, by Prop.~\ref{weak_Ramanujan}, almost every local factor is of the form $\Pi_v\cong\chi_1\times\chi_2\rtimes\sigma$ with unitary unramified characters $\chi_1,\chi_2,\sigma$. Its local lift $\tilde{\Pi}_v$ is the parabolically induced $\GL(4,\IA)$-representation
\begin{equation*}
\tilde{\Pi}_v\ \cong\ \chi_1\chi_2\sigma\times\chi_1\sigma\times\chi_2\sigma\times\sigma\ ,
\end{equation*}
\cite[Prop.\,2.5]{Asgari_Shahidi_GSp4_GL4}. Therefore, the lift $\Pi_v\mapsto\tilde{\Pi}_v$ commutes with local character twists at almost every place. Strong multiplicity one for $\GL(4)$ implies the statement.
\end{proof}

\begin{theorem}\label{thm:stable_para_implies_glob_generic} Suppose $\Pi=\bigotimes_{v}\Pi_v$  is a paramodular unitary cuspidal irreducible automorphic representation of $\GSp(4,\IA_{\IQ})$ that is not CAP nor weak endoscopic. We assume that $\Pi_\infty$ is in the discrete series, that the local factor $\Pi_2$ is spherical and that Hypothesis~\ref{conj:non-vanishing-L-value} holds.
Then $\Pi_v$ is locally generic at all nonarchimedean places $v$.
\end{theorem}

\begin{proof} Denote by $\Pi_\infty^W$ the generic constituent of the archimedean $L$-packet attached to $\Pi_\infty$.
By Prop.\,\ref{prop:mult_in_4D-reps} the irreducible representation $\Pi'=\Pi_\infty^W\otimes\Pi_\mathrm{fin}$ is cuspidal automorphic again.
By Prop.\,\ref{prop:weak_lift_to_globally_generic} there is a globally generic cuspidal automorphic irreducible representation $\Pi_{\mathrm{gen}}$, weakly equivalent to $\Pi$, such that $\Pi_{\mathrm{gen},\infty}$ is isomorphic to $\Pi_\infty^W$.

Since $\Pi'$ and $\Pi_{\mathrm{gen}}$ are weakly equivalent, there is a finite set $S$ of non-archimedean places such that $\Pi_{\mathrm{gen},v}$ is isomorphic to $\Pi_{v}$ for every place outside of $S$. This implies
\begin{equation}\label{eq:L-factor_local_global}
\frac{L_{\mathrm{PSS}}(\Pi',s)}{L_{\mathrm{PSS}}(\Pi_{\mathrm{gen}},s)}=\prod_{v\in S}\frac{L_{\mathrm{PSS}}(\Pi_v,s)}{L_{\mathrm{PSS}}(\Pi_{\mathrm{gen},v},s)}\ .
\end{equation}

At the place $v=2$, the local factor $\Pi_2$ is spherical by assumption and therefore locally generic by Prop.\,\ref{weak_Ramanujan}.
Now suppose there is at least one non-archimedean place $w$ of odd residue characteristic where the unitary local irreducible admissible representation $\Pi_w$ is non-generic. By Prop.\,\ref{prop:local_params_have_pole_at_one_half}, the right hand side of \eqref{eq:L-factor_local_global} must have an arithmetic progression $(s_k)_{k\in\IZ}$ of poles $s_k=1/2+i(\alpha+k\beta)$ with $\beta=2\pi/\ln (p_w)$ and some real $\alpha$.
Since $\Pi'$ is not CAP, $L_{\mathrm{PSS}}(\Pi',s)$ is holomorphic and this implies $L_{\mathrm{PSS}}(\Pi_{\mathrm{gen}},s_k)=0$ for every $k$ by \eqref{eq:L-factor_local_global}.
Hence, the partial $L$-function $L^S(\Pi,s_k)$ vanishes at every $s_k$ for sufficiently large $S$.

If Hypothesis~\ref{conj:non-vanishing-L-value} is true, there is a unitary idele class character $\mu$ of $\IQ^\times\backslash \IA^\times$ with $\mu_w = 1$
and there is some $k\in\IZ$ with $L^S(\Pi,\mu,s_k)\neq0$ for the partial $L$-factor outside of $S$.
Since $(\mu\circ\simi)\otimes\Pi$ is again cuspidal automorphic, the above argument gives $L^S(\Pi,\mu,s_k)=0$ for every $k$. This is a contradiction.

Therefore $\Pi$ is locally generic at every non-archimedean place $v$.
\end{proof}
Conversely, a generic non-archimedean $\Pi_v$ is always paramodular \cite{Roberts-Schmidt}.
In the situation of the theorem, $\Pi'$ is locally generic everywhere, so a result of Jiang and Soudry \cite[Corollary]{Jiang_Soudry_mult_one_GSp4} implies
\begin{equation}\label{eq:non-CAP-non-endo_param_is_glob_gen} \Pi'\cong\Pi_{\mathrm{gen}}. \end{equation}

\subsection{Weak endoscopic lift}
A cuspidal automorphic irreducible representation $\Pi$ of $\GSp(4,\IA)$, not CAP, is a weak endoscopic lift if there is a pair of cuspidal automorphic irreducible representations $\sigma_1,\;\sigma_2$ of $\GL(2,\IA)$ with the same central character, and local spinor $L$-factor
\begin{equation*} L_{\mathrm{PSS}}(\Pi_v,s)=L(\sigma_{1,v},s)L(\sigma_{2,v},s)\,.\end{equation*}
at almost every place \cite[\S5.2]{Weissauer200903}.
\begin{proposition}\label{prop:paramod_global_endo_is_generic}
Suppose a paramodular cuspidal irreducible automorphic representation $\Pi\cong\otimes_v\Pi_v$ of $\GSp(4,\IA_\IQ)$ is a weak endoscopic lift with local archimedean factor $\Pi_\infty$ in the discrete series. Then $\Pi$ is globally generic.
\end{proposition}
\begin{proof}
The automorphic representations $\sigma_1$ and $\sigma_2$ are locally tempered (Deligne), so the local endoscopic lift $\Pi_v$ is also tempered \cite[\S4.11]{Weissauer200903}.
At every non-archimedean place $v$ the local representation $\Pi_v$ is paramodular, hence locally generic by Lemma \ref{RS-temps_are_param_iff_gen}.
The archimedean factor $\Pi_\infty$ is also generic\cite[Thm.\,5.2]{Weissauer200903}. Hence $\Pi$ is globally generic \cite[Thm.\,4.1]{Weissauer200903}, \cite{Jiang_Soudry_mult_one_GSp4}.
\end{proof}

\section{Multiplicity one and strong multiplicity one}
We show the multiplicity one theorem and the strong multiplicity one theorem for paramodular cuspidal automorphic representations of $\GSp(4,\IA_\IQ)$ under certain restrictions. 
It is well-known that strong multiplicity one fails without the paramodularity assumption \cite{Howe_Piatetskii}, \cite{File_Strong_Mult_One}.
\begin{lemma}\label{lem:Soudry_lifts_not_paramodular} A cuspidal automorphic irreducible representation $\Pi$ of $\GSp(4,\IA_\IQ)$, that is CAP and strongly associated to the Klingen parabolic subgroup, is never paramodular.
\end{lemma}
\begin{proof}
Every such representation is a theta lift $\Pi=\theta(\sigma)$ of an automorphic representation $\sigma$ of $\GO_T(\IA_\IQ)$ for an anisotropic binary quadratic space $T$ over $\IQ$, see Soudry \cite{Soudry-Lift}.
Let $\mathrm{d}_T$ be the discriminant of $T$, then $T$ is isomorphic to $(K,t\cdot\mathrm{N}_K)$ for the norm $\mathrm{N}_K$ of the quadratic field $K=\IQ(\sqrt{-\mathrm{d}_T})$ and a squarefree integer $t$.
For a non-archimedean place $v$ of $\IQ$ that ramifies in $K$, the norm form on $K_w=K\otimes\IQ_v$ remains anisotropic. The local Weil representation of $\GO_T(\IQ_v)\times \GSp(4,\IQ_v)$ on the space of Schwartz-Bruhat functions $\mathcal{S}(K_w\times K_w\times\IQ_v^\times)$ does not admit non-zero paramodular invariant vectors; this can be shown by an elementary calculation. But then the global Weil representation does not admit non-zero paramodular invariants either. Since paramodular groups are compact, the functor of taking invariants is exact and therefore the paramodular invariant subspace of $\Pi=\theta(\sigma)$ is zero.
\end{proof}

\begin{theorem}[Multiplicity one] \label{thm:multiplicity_one} Suppose $\Pi=\Pi_\infty\otimes\Pi_{\mathrm{fin}}$ is a paramodular cuspidal automorphic irreducible representation of $\GSp(4,\IA)$ with $\Pi_\infty$ in the discrete series and such that $\Pi_v$ is spherical at the place $v=2$. 
Assume that Hypothesis \ref{conj:non-vanishing-L-value} holds. Then $\Pi$ occurs in the cuspidal spectrum with multiplicity one.
\end{theorem}

\begin{proof} 
If $\Pi$ is neither CAP nor weak endocopic, let $\Pi'=\Pi_\infty^W\otimes\Pi_{\mathrm{fin}}$ as above, then $\Pi'$ is globally generic by \eqref{eq:non-CAP-non-endo_param_is_glob_gen} and hence satisfies multiplicity one, see Jiang and Soudry \cite{Jiang_Soudry_mult_one_GSp4}. 
By Prop.~\ref{prop:mult_in_4D-reps} the multiplicity of $\Pi$ in the cuspidal spectrum is then also one.
If $\Pi$ is weak endoscopic, it occurs in the cuspidal spectrum with multiplicity one \cite[Thm.\,5.2]{Weissauer200903}.

Suppose $\Pi$ is CAP. By Lemma~\ref{lem:Soudry_lifts_not_paramodular} we can assume that $\Pi$ is strongly associated to the Borel or to the Siegel parabolic. Then it is a Saito-Kurokawa lift in the sense of Piatetskii-Shapiro, and thus occurs with multiplicity one \cite{PS-SaitoK}, \cite[(5.10)]{Gan_SK}.
\end{proof}

\begin{theorem}[Strong multiplicity one]\label{thm:strong_mult_one}
Suppose two paramodular automorphic cuspidal irreducible representations $\Pi_1$, $\Pi_2$ of $\GSp(4,\IA_\IQ)$ are locally isomorphic at almost every place. 
Assume that the archimedean local factors are either both in the generic discrete series or both in the holomorphic discrete series of $\GSp(4,\IR)$.
Assume that the local factors at the place $v=2$ are spherical and that Hypothesis \ref{conj:non-vanishing-L-value} holds. Then $\Pi_1$ and $\Pi_2$ are globally isomorphic.
\end{theorem}
\begin{proof}
Suppose $\Pi_1$ and $\Pi_2$ are not CAP. After possibly replacing the archimedean factor by the generic constituent in its local $L$-packet, we can assume that both $\Pi_1$ and $\Pi_2$ are globally generic by \eqref{eq:non-CAP-non-endo_param_is_glob_gen} and Proposition~\ref{prop:paramod_global_endo_is_generic}. 
Strong multiplicity one for globally generic representations has been shown by Soudry \cite[Thm.~1.5]{Soudry_Strong_Mult_One}.

If $\Pi_1$ and $\Pi_2$ are both CAP, they are strongly associated to the Borel or Siegel parabolic by Lemma~\ref{lem:Soudry_lifts_not_paramodular} and both occur via Saito-Kurokawa lifts in the sense of Piatetskii-Shapiro \cite{PS-SaitoK}.
For every local place $v$, the local factors $\Pi_{1,v}$ and $\Pi_{2,v}$ are non-generic and belong to the same Arthur packet. At the non-archimedean places, they are non-tempered by Lemma~\ref{RS-temps_are_param_iff_gen} and therefore isomorphic to the unique non-tempered constituent in the packet.
By assumption, the archimedean factors are in the discrete series, so they are isomorphic to the unique discrete series constituent of the archimedean Arthur packet.
\end{proof}

\appendix
\begin{footnotesize}
\bibliographystyle{amsalpha}

\end{footnotesize}

\vskip 20 pt
\begin{footnotesize}
\centering{Mirko R\"osner\\ Mathematisches Institut, Universit\"at Heidelberg\\ Im Neuenheimer Feld 205, 69120 Heidelberg\\ email: mroesner@mathi.uni-heidelberg.de}

\vskip 10 pt
\centering{Rainer Weissauer\\ Mathematisches Institut, Universit\"at Heidelberg\\ Im Neuenheimer Feld 205, 69120 Heidelberg\\ email: weissauer@mathi.uni-heidelberg.de}

\end{footnotesize}

\end{document}